\documentclass[a4paper]{article}

\usepackage[utf8]{inputenc}
\usepackage[english]{babel}
\usepackage{amsmath, amsfonts, amssymb, amsthm, stmaryrd}
\usepackage[colorlinks = true,
            linkcolor = black,
            urlcolor = black,
            bookmarksopen = true]{hyperref}
\usepackage[capitalise]{cleveref}
\usepackage{algorithm, algorithmic}

\usepackage{graphicx}
\graphicspath{{./img/}}

\theoremstyle{definition}
\newtheorem{theo}{Theorem}[section]
\crefname{theo}{Theorem}{Theorems}
\newtheorem{defi}{Definition}[section]

\newtheorem{es}{Example}[section]
\newtheorem{rem}{Remark}[section]
\numberwithin{equation}{section}

\title{On the cubic Pell equation over finite fields}
\author{Simone Dutto*, Nadir Murru**\\
* Politecnico di Torino, Department of Mathematics \\simone.dutto@polito.it \\
** Università degli Studi di Trento, Department of Mathematics \\ nadir.murru@unitn.it}
\date{}

\begin{document}

\maketitle

\begin{abstract}
The classical Pell equation can be extended to the cubic case considering the elements of norm one in $\mathbb Z[\sqrt[3]{r}]$, which satisfy
\begin{equation*}
    x^3 + r y^3 + r^2 z^3 - 3 r x y z = 1.
\end{equation*}
The solution of the cubic Pell equation is harder than the classical case, indeed a method for solving it as Diophantine equation is still missing \cite{Bar03}.
In this paper, we study the cubic Pell equation over finite fields, extending the results that hold for the classical one.
In particular, we provide a novel method for counting the number of solutions in all possible cases depending on the value of $r$.
Moreover, we are also able to provide a method for generating all the solutions.
\end{abstract}

\section{Introduction}
\label{sec:intro}

The Pell equation 
\begin{equation*}
    x^2 - d y^2 = 1,
\end{equation*}
is an important and well studied Diophantine equation, for $d$ a non--square positive integer.
Finding its solutions is equivalent to finding the elements of $\mathbb{Z}[\sqrt{d}]$ of norm one.
There are well known methods for solving this equation.
They are mainly based on continued fractions that allow to find a fundamental solution, which is then used for generating all the other ones.
Currently, there are still several important issues regarding the Pell equation.
For instance, the study of the size of the fundamental solution is an interesting problem addressed in several papers, e.g., \cite{Bou15, Fou16, Xi18}.
Recently, the solvability of simultaneous Pell equations and explicit formulas for their solutions have been also studied in \cite{HPT15, Cip18, FY21}.
Moreover, it is also interesting to study the Pell equation over finite fields, determining the number of solutions and their properties \cite{MV92, TOK10, Tec11, Coh21}.
For further information about the importance of the study of the Pell equation see, e.g., \cite{JW09}.

Thus, it is natural to consider generalizations of the Pell equation, starting from the cubic case.
Considering the connection between the Pell equation and the elements of norm one in a quadratic field, the analogue of the Pell equation in the cubic case is given by the equation
\begin{equation*}
    x^3 + r y^3 + r^2 z^3 - 3 r x y z = 1,
\end{equation*}
where $r$ is a cube--free integer, i.e., we are asking for the elements of norm one in $\mathbb{Z}[\sqrt[3]{r}]$.
First studies of the cubic Pell equation can be found in \cite{Mat89} and \cite{Wol23}.
In \cite{Dau29}, the author proposed a method for solving the cubic Pell equation by means of a generalization of continued fractions due to Jacobi \cite{Jac91}.
However, this method is not always useful for this purpose, since the periodicity of the Jacobi algorithm is still a fascinating open problem for all cubic irrationals.
This question was also addressed, e.g., in \cite{Ber74} and \cite{Ber75}.
The solutions of the cubic Pell equation were studied in \cite{Bal99} from the point of view of recurrent sequences, since Lucas sequences are solutions, up to constants, of the classical Pell equation.
In general, the cubic Pell equation is very hard to solve for any cube--free $r$.
In \cite{Bar03}, the author exhibited an algorithm for finding the fundamental solutions of the cubic Pell equation that works only in some cases.
Thus, the problem of solving the cubic Pell equation is still open.
For more motivation and results about the cubic Pell equation see also \cite{HW18}.

In this paper, we address the problem of solving the cubic Pell equation over finite fields.
In particular, in \cref{sec:qua} we recall the classical Pell equation and its definition as the elements of norm one in a quadratic field.
We also introduce a particular parameterization using a projectivization, which is useful for studying the Pell equation over finite fields and it is also handy for a generalization in the cubic case.
In \cref{sec:ffcon}, we recall the structure of the Pell conic over finite fields obtaining also the results of \cite{MV92} in a different way.
\cref{sec:cub} is devoted to the introduction of the cubic Pell equation.
Here, we also introduce its parameterization that allows to study its structure over finite fields, giving also methods for generating the solutions.
Finally, in \cref{sec:ffcub}, we describe the behavior of the cubic Pell equation over finite fields.

\section{The Pell equation} 
\label{sec:qua}

The classical  Pell equation is the equation of the form
\begin{equation*}
    x^2 - d y^2 = 1,
\end{equation*}
where $d$ is a positive square--free integer and solutions are sought for $(x, y) \in \mathbb{Z}^2$.

In this work we consider the Pell equation in general terms, considering any element $d$ in a field $\mathbb{F}$ and taking the polynomial ring
\begin{align*}
  \mathcal{R}_d 
  := \mathbb{F}[t] / \langle t^2 - d \rangle,
\end{align*}
which inherits from the polynomial product the operation
\begin{equation}
\label{eq:brah}
  (x_1 + y_1 t) \cdot (x_2 + y_2 t) 
  = (x_1 x_2 + d y_1 y _2) + (x_1 y_2 + y_1 x_2) t.
\end{equation}
The conjugate of an element $x + y t \in \mathcal{R}_d$ is defined as $x - y t$.
The product of an element with its conjugate defines the norm
\begin{align*}
  N_d(x + y t) 
  := (x + y t) \cdot (x - y t) 
  = x^2 - d y^2 \in \mathbb{F}.
\end{align*}
The unitary elements of $\mathcal{R}_d$ with respect to the norm $N_d$
\begin{align*}
  \mathcal{U}(\mathcal{R}_d) 
  := \{x + y t \in \mathcal{R}_d \,|\, N_d(x + y t) = 1\}, 
\end{align*}
form a commutative group that is clearly isomorphic to the \emph{Pell conic}
\begin{equation*}
    \mathcal{C}_d := \{(x, y) \in \mathbb{F}^2 \,|\, x^2 - d y^2 = 1\},
\end{equation*}
equipped with the classical Brahmagupta product
\begin{equation*}
    (x_1, y_1) \otimes_d (x_2, y_2) := (x_1x_2 + d y_1y_2, x_1y_2+y_1x_2).
\end{equation*}
Due to this group isomorphism, in the following we will use $\otimes_d$ also for denoting the product over $\mathcal R_d$, so that its dependence on $d$ is highlighted.
The operation $\otimes_d$ over the Pell conic has a geometrical interpretation, since the resulting point is the intersection between the curve and the line through the identity point $(1, 0)$ parallel to that passing through the initial two points.
This interpretation is analogous to that of the elliptic curve operation (see, e.g., \cite{BMD21}).

In order to introduce a parameterization for the Pell conic, we need to characterize the set of the elements of $\mathcal{R}_d$ that are invertible with respect to $\otimes_d$:
\begin{enumerate}
    \item if $d$ is a non--square element in $\mathbb{F}$, then the invertible elements of $\mathcal{R}_d$ are
    \begin{equation*}
        \mathcal{R}_d^{\otimes_d} = \mathcal{R}_d \smallsetminus \{0\};
    \end{equation*}
    \item if $d$ is a square in $\mathbb{F}$ and $s$ is a fixed square root of $d$, then they are
    \begin{equation*}
    \mathcal{R}_d^{\otimes_d} = \mathcal{R}_d \smallsetminus \{0, \pm s y + y t \,|\, y \in \mathbb{F}\}.
    \end{equation*}
\end{enumerate}

\begin{defi}
\label{def:dproj}
The projectivization of $\mathcal{R}_d^{\otimes_d}$ is
\begin{equation}
\label{eq:conproj}
\begin{aligned}
  \mathbb{P}_d 
  & := \big\{[m : n] = \{\lambda (m + n t) \,|\, \lambda \in \mathbb{F}^\times\} \,|\, m + n t \in \mathcal{R}_d^{\otimes_d}\big\} \\
  & = \begin{cases}
    \big\{[m : 1], [1 : 0] \,|\, m \in \mathbb{F}\big\}, & \text{if } d \text{ is a non--square,} \\
    \big\{[m : 1], [1 : 0] \,|\, m \in \mathbb{F} \smallsetminus \{\pm s\}\big\}, & \text{otherwise}
  \end{cases} \\
  & \cong \begin{cases}
    \mathbb{F} \cup \{\alpha\}, & \text{if } d \text{ is a non--square,} \\
    (\mathbb{F} \smallsetminus \{\pm s\}) \cup \{\alpha\}, & \text{otherwise,}
  \end{cases}
\end{aligned}
\end{equation}
where $\alpha$ denotes the point at infinity and, in case of $d$ square, $s \in \mathbb{F}$ is a fixed square root of $d$.
Since the Brahmagupta product $\otimes_d$ consists of homogeneous polynomials, it is well defined also on $\mathbb{P}_d$ and determines a commutative group with identity $[1 : 0]$ and inverse of $[m : n]$ given by $[m : - n]$.
\end{defi}

This projectivization is actually a parameterization of the Pell conic, which is useful for studying some of its properties over finite fields and will be naturally generalized also for the cubic case.
The following theorem provides an explicit group isomorphism between $(\mathbb P_d, \otimes_d)$ and $(\mathcal C_d, \otimes_d)$.
The result was introduced in \cite{BCM10}, here we give a different formulation and a proof that can be adapted to the cubic case.

\begin{theo}
\label{th:conphi}
Given $d \in \mathbb{F}$, there is the group isomorphism
\begin{align*}
  \phi : \big(\mathbb{P}_d, \otimes_d\big) 
  & \xrightarrow{\;\sim\;} \big(\mathcal{C}_d, \otimes_d\big), \\
  [m : n] 
  & \longmapsto \frac{(m, n)^{\otimes_d 2}}{N_d(m, n)} 
  = \left(\frac{m^2 + d n^2}{m^2 - d n^2}, \frac{2 m n}{m^2 - d n^2}\right).
\end{align*}
\end{theo}

\begin{proof} 
In order for $\phi$ to be a group isomorphism, it must be:
\begin{itemize}
\item well defined:
for any $[m : n] \in \mathbb{P}_d$, $\lambda \in \mathbb{F}^\times$
\begin{align*}
  \phi([\lambda m : \lambda n]) 
  = \frac{\lambda^2 (m, n)^{\otimes_d 2}}{\lambda^2 N_d(m, n)} 
  = \phi([m : n]),
\end{align*}
and $\phi(\mathbb{P}_d) \subseteq \mathcal{C}_d$ since for any $[m : n] \in \mathbb{P}_d$
\begin{align*}
  N_d(\phi([m : n])) 
  = \frac{N_d(m, n)^2}{N_d(m, n)^2} 
  = 1;
\end{align*}

\item a group homomorphism:
for any $[m_1 : n_1], [m_2 : n_2] \in \mathbb{P}_d$,
\begin{align*}
  \phi([m_1 : n_1] \otimes_d [m_2 : n_2]) 
  & = \frac{([m_1 : n_1] \otimes_d [m_2 : n_2])^{\otimes_d 2}}{N_d([m_1 : n_1] \otimes_d [m_2 : n_2])} \\
  & = \frac{(m_1, n_1)^{\otimes_d 2} \otimes_d (m_2, n_2)^{\otimes_d 2}}{N_d(m_1, n_1) N_d(m_2, n_2)} \\
  & = \phi([m_1 : n_1]) \otimes_d \phi([m_2 : n_2]);
\end{align*}

\item injective:
for any $[m, n] \in \mathbb{P}_d$,
\begin{align*}
  \phi([m : n]) = (1, 0) 
  & \Leftrightarrow \begin{cases}
  m^2 - d n^2 = m^2 + d n^2, \\
  0 = 2 m n \end{cases} \\
  & \Leftrightarrow n = 0 
  \Leftrightarrow \text{ker}(\phi) = \{[1 : 0]\};
\end{align*}

\item surjective:
if $(x, 0) \in \mathcal{C}_d$, then $0 = 2 m n$ so that
\begin{align*}
  m = 0 \Rightarrow x = \frac{d n^2}{- d n^2} = - 1,
  \text{ or }
  n = 0 \Rightarrow x = \frac{m^2}{m^2} = 1,
\end{align*}
i.e., $\phi([0 : 1]) = (- 1, 0)$ and $\phi([1 : 0]) = (1, 0)$, while when $y \neq 0$ then $d = \frac{x^2 - 1}{y^2}$ and
\begin{align*}
  \begin{cases}
  x = \frac{m^2 y^2 + (x^2 - 1) n^2}{m^2 y^2 - (x^2 - 1) n^2}, \\
  y = \frac{2 m n y^2}{m^2 y^2 - (x^2 - 1) n^2} \end{cases} 
  \Rightarrow \begin{cases}
  m^2 y^2 - 2 m n x y + n^2 (x^2 - 1) = 0, \\
  m^2 y^2 - n^2 (x^2 - 1) = 2 m n y \end{cases}
\end{align*}
\begin{align*}
  \Rightarrow m_\pm 
  = \frac{n x y \pm \sqrt{n^2 x^2 y^2 - (x^2 - 1) n^2 y^2}}{y^2} 
  = \frac{n x y \pm n y}{y^2} = n \frac{x \pm 1}{y},
\end{align*}
but $m_-$ in the second equation returns
\begin{align*}
  n^2 \frac{(x - 1)^2}{y^2} y^2 - n^2 (x - 1) (x + 1) 
  = 2 n^2 \frac{x - 1}{y} y
  & \Rightarrow (x - 1) - (x + 1) = 2 \\
  & \Rightarrow - 2 = 2,
\end{align*}
thus, $m_+$ is the only option and $\phi([x + 1 : y]) = (x, y)$.
\end{itemize}
In conclusion, $\phi$ is a group isomorphism.
\end{proof}

In this proof we also constructed the inverse of $\phi$ that is
\begin{align*}
  \phi^{-1} : \big(\mathcal{C}_d, \otimes_d\big) 
  & \xrightarrow{\;\sim\;} \big(\mathbb{P}_d, \otimes_d\big), \\
  (-1, 0) 
  & \longmapsto [0 : 1], \\
  (x, y) 
  & \longmapsto [x + 1 : y].
\end{align*}
From a geometrical point of view, when taking the canonical representative of the image, the first entry results in the slope of the line through $(-1, 0)$ and $(x, y)$ evaluated with $x$ depending on $y$, except for the image of $(1, 0)$ which is the point at infinity $\alpha = [1 : 0]$.

The group isomorphism $\phi$ gives also a direct method to generate all the solutions of the Pell equation $x^2 - d y^2 = 1 \in \mathbb{F}$ from the elements of $\mathbb{P}_d$, which requires half the size to be stored, since
\begin{align*}
  \phi(\alpha) & = (1, 0), \\ 
  \phi(0) & = (- 1, 0), \\
  \phi(m) & = \left(\frac{m^2 + d}{m^2 - d}, \frac{2 m}{m^2 - d}\right), \quad \text{for } m \neq \alpha, 0.
\end{align*}

\section{The Pell conic over finite fields}
\label{sec:ffcon}

When $\mathbb{F} = \mathbb{F}_q$ with $q = p^k$ and $p$ prime, the group structure of the Pell conic depends on the parameter $d$ being or not a square.
These situations are fully described by Menezes and Vanstone \cite{MV92} giving also the order of the Pell conic in these two cases, i.e., the number of solutions of the Pell equation over finite fields.
In \cref{ssec:dnres,ssec:dres}, we report these results including the proofs and proving them also in an alternative way connected to the previous parameterization.
In this way, we provide the ideas that will be exploited for studying the cubic Pell equation over finite fields in \cref{sec:ffcub}.

\subsection{$d$ non--square}
\label{ssec:dnres}

When $d$ is a not a square, $t^2 - d \in \mathbb{F}_q[t]$ is irreducible over $\mathbb{F}_q$, so that
\begin{align*}
  \mathcal{R}_d = \mathbb{F}_q[t] / \langle t^2 - d \rangle \cong \mathbb{F}_{q^2}.
\end{align*}

\begin{theo}
\label{th:mv1}
If $d$ is a non--square in $\mathbb{F}_q$, then $(\mathcal{C}_d, \otimes_d)$ is a cyclic group of order $q + 1$ \cite{MV92}.
\end{theo}

\begin{proof}
We have that $\mathcal{R}_d^{\otimes_d} \cong \mathbb{F}_{q^2}^\times$ has $q^2 - 1$ elements.
If $G \subset \mathbb{F}_{q^2}^\times$ denotes the multiplicative subgroup of order $q + 1$, then $x + y t \in G \Leftrightarrow (x + y t)^{q + 1} = 1$ and
\begin{align*}
  (x + y t)^{q + 1} 
  & = (x + y t)^q (x + y t) \\
  & = (x + y t^q) (x + y t) \\
  & = \big(x + y (t^2)^{(q - 1)/2} t \big) (x + y t) \\
  & = \big(x + y d^{(q - 1)/2} t\big) (x + y t) \\
  & = (x - y t) (x + y t) \\
  & = x^2 - d y^2,
\end{align*}
so that $x + y t \in G \Leftrightarrow (x, y) \in \mathcal{C}_d$.
This association is a group isomorphism between $G$ and $(\mathcal{C}_d, \otimes_d)$, hence the Pell conic is a cyclic group of order $q + 1$.
\end{proof}

Looking at the projectivization $\mathbb{P}_d$, since there are no square roots of $d$ in $\mathbb{F}_q$, then $\#\mathbb{P}_d = q + 1$ from \cref{eq:conproj}.
This is confirmed also considering
\begin{align*}
  \big(\mathbb{P}_d, \otimes_d\big) 
  \cong \mathcal{R}_d^{\otimes_d} / \mathbb{F}_q^\times 
  \cong \mathbb{F}_{q^2}^\times / \mathbb{F}_q^\times,
\end{align*}
which proves also that $(\mathbb{P}_d, \otimes_d)$ is cyclic because quotient of cyclic groups.
Thus, using the group isomorphism $\phi$ obtained for a general field in \cref{th:conphi} also proves that $(\mathcal{C}_d, \otimes_d)$ is cyclic of order $q + 1$.
In addition, $\phi$ allows also to describe each point of the conic with half the size with respect to the group isomorphism obtained in \cref{th:mv1}.

\subsection{$d$ square}
\label{ssec:dres}

If we suppose $d$ is a square, fixed a square root $s$ of $d$, then $\pm s \in \mathbb{F}_q$ and  $\mathcal{R}_d$ is a ring.
As in the previous case, the Pell conic is cyclic because of the following result.

\begin{theo}
\label{th:mv2}
If $d$ is a square in $\mathbb{F}_q$, then $(\mathcal{C}_d, \otimes_d)$ is a cyclic group of order $q - 1$ \cite{MV92}.
\end{theo}

\begin{proof}
Fixed a square root $s \in \mathbb{F}_q$ of $d$, the norm of a point $(x, y) \in \mathcal{C}_d$ can be factorized as
\begin{align*}
  1 = x^2 - d y^2 
  = (x - s y) (x + s y) 
  = u v,
\end{align*}
so that
\begin{align*}
  x = \frac{v + u}{2}, \quad y 
  = \frac{v - u}{2 s},
\end{align*}
which results in a bijective correspondence between $(x, y) \in \mathcal{C}_d$ and $(u, v) \in \mathbb{F}_q^2$ such that $u v = 1$.
This equation has exactly $q - 1$ solutions in $\mathbb{F}_q^2$ and, in particular, a unique solution for each $u \in \mathbb{F}_q^\times$.
Thus, 
\begin{align*}
  \big(\mathcal{C}_d, \otimes_d\big) 
  & \;\cong\; \mathbb{F}^\times \\
  (x, y) 
  & \longmapsto x - s y, \\
  \left(\frac{1 + u^2}{2 u}, \frac{1 - u^2}{2 s u}\right) 
  & \longmapsfrom u,
\end{align*} 
is bijective and a group homomorphism, i.e., $(\mathcal{C}_d, \otimes_d)$ is cyclic of order $q - 1$.
\end{proof}

When considering the projectivization $\mathbb{P}_d$, from \cref{eq:conproj}, $\#\mathbb{P}_d = q - 1$.
This is confirmed by the following result.

\begin{theo}
\label{th:conpr}
If $d$ is a square in $\mathbb{F}_q$, then $(\mathbb{P}_d, \odot_d)$ is a cyclic group of order $q - 1$.
\end{theo}

\begin{proof}
Fixed $s$ square root of $d$ in $\mathbb{F}_q$, $t^2 - d$ is reducible over $\mathbb{F}_q$ as
\begin{align*}
  t^2 - d = (t - s) (t + s),
\end{align*}
so that, using the Chinese remainder theorem, there is the ring isomorphism
\begin{align*}
  \mathcal{R}_d = \mathbb{F}_q[t] / \langle t^2 - d \rangle
  & \xrightarrow{\;\sim\;} \mathbb{F}_q[t] / \langle t - s \rangle \times \mathbb{F}_q[t] / \langle t + s \rangle, \\
  x + y t
  & \longmapsto (x + s y, x - s y).
\end{align*}
In addition, $\mathbb{F}_q[t] / \langle t - s \rangle \cong \mathbb{F}_q[t] / \langle t + s \rangle \cong \mathbb{F}_q$, so that when passing to the quotients there is the group isomorphism
\begin{align*}
  \big(\mathbb{P}_d, \otimes_d\big)
  \cong \mathcal{R}_d^{\otimes_d} / \mathbb{F}_q^\times 
  & \;\cong\; (\mathbb{F}_q^\times \times \mathbb{F}_q^\times) / \mathbb{F}_q^\times
  \cong \mathbb{F}_q^\times, \\
  [m : n] = \{\lambda (m + n t) \,|\, \lambda \neq 0\}
  & \longmapsto \frac{m - s n}{m + s n}, \\
  \left[s (1 + u) : 1 - u\right]
  & \longmapsfrom u,
\end{align*}
which confirms that $\big(\mathbb{P}_d, \otimes_d\big)$ is a cyclic group of order $q - 1$.
\end{proof}

In particular, the combination of the group isomorphisms obtained in \cref{th:mv2,th:conpr} is
\begin{alignat*}{5}
  & \big(\mathbb{P}_d, \otimes_d\big) 
  && \xrightarrow{\;\sim\;} \mathbb{F}_q^\times 
  && \xrightarrow{\;\sim\;} \big(\mathcal{C}_d, \otimes_d\big), \\
  & [m : n] 
  && \longmapsto \frac{m - s n}{m + s n} 
  && \longmapsto \left(\frac{1 + \left(\frac{m - s n}{m + s n}\right)^2}{2 \frac{m - s n}{m + s n}},
  \frac{1 - \left(\frac{m - s n}{m + s n}\right)^2}{2 s \frac{m - s n}{m + s n}}\right) = (x, y),
\end{alignat*}
where
\begin{align*}
(x, y) 
& = \left(\frac{(m + s n)^2 + (m - s n)^2}{2 (m - s n) (m + s n)}, 
\frac{(m + s n)^2 - (m - s n)^2}{2 s (m - s n) (m + s n)}\right) \\
& = \left(\frac{2 m^2 + 2 d n^2}{2 (m^2 - d n^2)}, 
\frac{4 s m n}{2 s (m^2 - d n^2)}\right) 
= \left(\frac{m^2 + d n^2}{m^2 - d n^2}, \frac{2 m n}{m^2 - d n^2}\right).
\end{align*}
The inverse is given by 
\begin{alignat*}{5}
  & \big(\mathcal{C}_d, \otimes_d\big) 
  && \xrightarrow{\;\sim\;} \mathbb{F}_q^\times 
  && \xrightarrow{\;\sim\;} \big(\mathbb{P}_d, \otimes_d\big), \\
  & (x, y) 
  && \longmapsto x - s y 
  && \longmapsto [s (1 + x - s y) : 1 - x + s y] = [m : n],
\end{alignat*}
where if $1 + x + s y \neq 0$ then
\begin{align*}
  [m : n] 
  & = [s (1 + x - s y) (1 + x + s y) : (1 - x + s y) (1 + x + s y)] \\
  & = [s (1 + 2 x + x^2 - d y^2) : 1 + 2 s y + d y^2 - x^2] \\
  & = [2 s (1 + x) : 2 s y] = [1 + x : y],
\end{align*}
while 
\begin{align*}
  \begin{cases}
  1 + x + s y = 0, \\
  x^2 - d y^2 = 1 \end{cases}
  \Rightarrow
  \begin{cases}
  x = - 1 - s y, \\
  1 + 2 s y + d y^2 - d y^2 = 1 \end{cases}
  \Rightarrow
  \begin{cases}
  x = - 1, \\
  y = 0, \end{cases}
\end{align*}
so that
\begin{align*}
  [m : n] = \begin{cases}
  [0 : 1], & \text{if } (x, y) = (- 1, 0), \\
  [1 + x : y], & \text{otherwise.} \end{cases}
\end{align*}
These are exactly $\phi$ and $\phi^{-1}$ obtained for a general field in \cref{th:conphi}, which allow to describe each point of the conic with half the size with respect to the group isomorphism obtained in \cref{th:mv2}.

\section{The cubic Pell equation} 
\label{sec:cub}

In this section, we introduce and study the cubic Pell equation in a way similar to the one used in \cref{sec:qua} for the quadratic case.
Then we approach the study of the Pell cubic equation over finite fields in the next section.

Given a field $\mathbb{F}$ and an element $r \in \mathbb{F}$, we consider the polynomial ring
\begin{align*}
  \mathcal{R}_r := \mathbb{F}[t] / \langle t^3 - r \rangle,
\end{align*}
which inherits from the polynomial product the operation
\begin{align*}
  (x_1 + y_1 t + z_1 t^2) \cdot (x_2 + y_2 t + z_2 t^2) 
  = &\;x_1 x_2 + r (y_1 z _2 + z_1 y_2) \\
  & \quad + (x_1 y_2 + y_1 x_2 + r z_1 z_2) t \\
  & \qquad + (x_1 z_2 + y_1 y_2 + z_1 x_2) t^2.
\end{align*}

Considering the cubic roots of unity $\{1, \omega, \omega^2\}$, we can define the conjugate of an element $x + y t + z t^2 \in \mathcal{R}_r$ as 
\begin{align*}
  (x + y \omega t + z \omega^2 t^2) \cdot (x + y \omega^2 t + z \omega t^2)
  = &\;(x^2 - r y z) \\
  & \quad + (r z^2 - x y) t \\
  & \qquad + (y^2 - x z) t^2.
\end{align*}
The product of an element with its conjugate defines the norm
\begin{align*}
  N_r(x + y t + z t^2) 
  & := (x + y t + z t^2) \cdot (x + y \omega t + z \omega^2 t^2) \cdot (x + y \omega^2 t + z \omega t^2) \\
  & = x^3 - 3 r x y z + r y^3 + r^2 z^3,
\end{align*}
that, as for the Pell conic, allows to provide a trivial group isomorphism between the unitary elements of $\mathcal{R}_r$ with respect to the norm $N_r$
\begin{align*}
 \mathcal{U}(\mathcal{R}_r) := \{x + y t + z t^2 \in \mathcal{R}_r \,|\, N_r(x + y t + z t^2) = 1\},
\end{align*}
and the \emph{Pell cubic}
\begin{align*}
 \mathcal{C}_r := \{(x, y, z) \in \mathbb{F}^3 \,|\, x^3 - 3 r x y z + r y^3 + r^2 z^3 = 1\},
\end{align*}
that, with the generalization of the Brahmagupta product
\begin{align*}
    (x_1, y_1, z_1) \odot_r (x_2, y_2, z_2) 
  := \big(&x_1 x_2 + r (y_1 z _2 + z_1 y_2), \\
  & \quad x_1 y_2 + y_1 x_2 + r z_1 z_2, \\
  & \qquad x_1 z_2 + y_1 y_2 + z_1 x_2\big),
\end{align*}
is a commutative group with identity $(1, 0, 0)$ and inverse of an element $(x, y, z)$ given by is its conjugate
\begin{equation*}
  (x^2 - r y z, r z^2 - x y, y^2 - x z).
\end{equation*}
Due to this group isomorphism, in the following we will use $\odot_r$ also for denoting the product over $\mathcal{R}_r$, in order to highlight the dependence on $r$ for this product.

As for \cref{def:dproj}, we can consider the set of the invertible elements of $\mathcal{R}_r$ with respect to $\odot_r$, denoted as
\begin{equation*}
    \mathcal{R}_r^{\odot_r} = \{x + y t + z t^2 \in \mathcal{R}_r \,|\, N_r(x + y t + z t^2) \neq 0\},
\end{equation*}
as well as introduce a parameterization for the Pell cubic with a projectivization.

\begin{defi}
We define the projectivization of $\mathcal{R}_r^{\odot_r}$ as
\begin{align*}
  \mathbb{P}_r 
  := \big\{[l : m : n] = \{\lambda (l + m t + n t^2) \,|\, \lambda \in \mathbb{F}^\times\} \,|\, l + m t + n t^2 \in \mathcal{R}_r^{\odot_r}\big\}.
\end{align*}
With the product $\odot_r$, it is a commutative group with identity $[1 : 0 : 0]$ and inverse of $[l : m : n]$ given by $[l^2 - r m n : r n^2 - l m : m^2 - l n]$.
\end{defi}

Differently from the quadratic case, the group isomorphism between  $(\mathbb{P}_r, \odot_r)$ and $(\mathcal{C}_r, \odot_r)$ is not easy to find.
However, it is still possible to exploit the projectivization to give a complete characterization of the Pell cubic over finite fields generalizing the results in \cref{sec:ffcon}, considering three cases:
\begin{enumerate}
\item if $r$ is not a cube in $\mathbb{F}$, then $N_r(x + y t + z t^2) \neq 0 \Leftrightarrow x + y t + z t^2 \neq 0$ and
\begin{equation}
\label{eq:cubproj1}
\begin{aligned}
  \mathbb{P}_r
  & = \big\{[l : m : 1], [l : 1 : 0], [1 : 0 : 0] \,|\, l, m \in \mathbb{F}\big\} \\
  & \cong (\mathbb{F} \times \mathbb{F}) \cup (\mathbb{F} \times \{\alpha\}) \cup \{(\alpha, \alpha)\},
\end{aligned}
\end{equation}
where $(\alpha, \alpha)$ denotes the point at infinity and $\mathbb{F} \times \{\alpha\}$ is a line at infinity;

\item if $r$ is a cube and $\{1, \omega, \omega^2\} \subset \mathbb{F}$, then $\mathbb{F}$ contains also all the cubic roots of $r$ that, when denoting one of them with $s$, are $\{s, s \omega, s \omega^2\}$.
In this case, $t^3 - r$ can be completely decomposed and the set of elements of norm zero is generated by three elements, i.e.,
\begin{align*}
  \big\{x + y t + z t^2 \in \mathcal{R}_r \,|\, N_r(x + y t + z t^2) = 0\big\} 
  & = \big\langle t - s, t - s \omega, t - s \omega^2 \big\rangle.
\end{align*}

Thus, $\big\langle [- s : 1 : 0], [- s \omega : 1 : 0], [- s \omega^2 : 1 : 0] \big\rangle \not\subset \mathbb{P}_r$ and this can help in obtaining an explicit form for its elements, like in \cref{eq:cubproj1}.

The elements generated by $[- s : 1 : 0]$ are, for any $l \in \mathbb{F}$,
\begin{align*}
  [- s : 1 : 0] \odot_r [l : 1 : 0] = [- l s : - s + l : 1],
\end{align*}
and, for any $l', m' \in \mathbb{F}$,
\begin{align*}
  [- s : 1 : 0] \odot_r [l' : m' : 1] 
  & = [- l' s + s^3 : - m' s + l' : - s + m'] \\
  & = \begin{cases}
  [- s (l' - s^2) : l' - s^2 : 0], & \text{if } m' = s, \\
  \left[- \left(\frac{l' - s^2}{m' - s}\right) s : \left(\frac{l' - s^2}{m' - s}\right) - s : 1\right], & \text{otherwise} \end{cases} \\
  & = \begin{cases}
  [- s : 1 : 0], & \text{if } m' = s, \\
  \left[- l s : l - s : 1\right], & \text{with } l = \frac{l' - s^2}{m' - s} \text{ otherwise.} \end{cases}
\end{align*}
We obtain analogous results with the second and third generators, with $s \omega$ and $s \omega^2$ instead of $s$, respectively.

Looking at the intersections between the three sets generated, we observe that if $0 \leq i < j \leq 2$, then $[- l s \omega^i : l - s \omega^i : 1] = [- l' s \omega^j : l' - s \omega^j : 1]$ if and only if
\begin{align*}
  \begin{cases}
  - l s \omega^i = - l' s \omega^j, \\
  l - s \omega^i = l' - s \omega^j \end{cases}
  \Leftrightarrow \begin{cases}
  l = l' \omega^{j - i}, \\
  l' \omega^{j - i} - l' = s \omega^i - s \omega^j \end{cases}
  \Leftrightarrow \begin{cases}
  l = - s \omega^j, \\
  l' = - s \omega^i, \end{cases}
\end{align*}
which means that 
\begin{align*}
  \big\langle (- s \omega^i, 1, 0) \big\rangle \cap\big\langle (- s \omega^j, 1, 0) \big\rangle 
  & = \big\{[s^2 \omega^{i + j} : - s (\omega^i + \omega^j) : 1]\big\} \\
  & = \big\{[s^2 \omega^{i + j} : s \omega^k : 1],\, k \neq i, j\big\}.
\end{align*}
Thus, when considering the union of the three sets obtained by the generators, three of all the elements of norm zero are obtained twice.
In particular, in $\langle[- s \omega^i : 1 : 0]\rangle$, they are those with second coordinate $m = s \omega^k$ with $k \neq i$.
Hence, one of the duplicates can be removed by excluding for each $i \in \{0, 1, 2\}$ the element with $m = s \omega^{i - 1}$, so that
\begin{align}
\label{eq:cubproj2}
  \mathbb{P}_r =\;
  & \big\{[l : m : 1], [l : 1 : 0], [1 : 0 : 0] \,|\, l, m \in \mathbb{F}\big\} \\
  & \smallsetminus\!\bigcup_{i \in \{0, 1, 2\}}\!\big\{[- s \omega^i\!:\!1\!:\!0], [- (m\!+\!s \omega^i) s \omega^i\!:\!m\!:\!1] \,|\, m \in \mathbb{F}\!\smallsetminus\!\{s \omega^{i - 1}\}\big\}.
  \nonumber
\end{align}

\item if $r$ is a cube and $\mathbb{F}$ does not contain any non--trivial cubic root of unity, i.e., $\{\omega, \omega^2\} \not\subset \mathbb{F}$, then only one root $s$ of $r$ is in $\mathbb{F}$, so that
\begin{align*}
  t^3 - r 
  = (t - s) \odot_r (t + s \omega) \odot_r (t + s \omega^2) 
  = (t - s) \odot_r (t^2 + s t + s^2),
\end{align*}
and the elements of norm zero are
\begin{align*}
  \big\{x + y t + z t^2 \in \mathcal{R}_r \,|\, N_r(x + y t + z t^2) = 0\big\} 
  & = \big\langle t - s, t^2 + s t + s^2 \big\rangle_{\odot_r} \\
  & \cong \big\langle (- s, 1, 0), (s^2, s, 1) \big\rangle_{\odot_r}.
\end{align*}
As before, when considering the projectivization, $[- s : 1 : 0]$ generates the elements $[- (m + s) s : m : 1]$ for $m \in \mathbb{F}$.
With the second generator no other element is added since 
\begin{align*}
  [s^2 : s : 1] \odot_r [- s : 1 : 0] = [0 : 0 : 0],
\end{align*}
while for every $[l : m : n]$ non multiple of $[- s : 1 : 0]$
\begin{align*}
  [s^2 : s : 1] \odot_r [l : m : n] 
  & = [s^2 (l\!+\!m s\!+\!n s^2) : s (l\!+\!m s\!+\!n s^2) : l\!+\!m s\!+\!n s^2] \\
  & = [s^2 : s : 1].
\end{align*}
In conclusion
\begin{equation}
\label{eq:cubproj3}
\begin{aligned}
  \mathbb{P}_r =\;
  & \big\{[l : m : 1], [l : 1 : 0], [1 : 0 : 0] \,|\, l, m \in \mathbb{F}\big\} \\
  & \smallsetminus \big\{[- s : 1 : 0], [- (m + s) s : m : 1], [s^2 : s : 1] \,|\, m \in \mathbb{F}\big\}.
\end{aligned}
\end{equation}
\end{enumerate} 

\section{The Pell cubic over finite fields}
\label{sec:ffcub}

In this section, we give a full description of the solutions of the cubic Pell equation when $\mathbb{F} = \mathbb{F}_q$ with $q = p^k$ and $p$ prime.
This characterization depends on the parameter $r \in \mathbb{F}_q$ and there are three different scenarios due to the value of $\gcd(3, q - 1)$ in the generalization of the Euler criterion:
\begin{align*}
  r \in \mathbb{F}_q \text{ is a cube}
  \Leftrightarrow r^{(q - 1) / \gcd(3, q - 1)} = 1.
\end{align*}

\subsection{$r$ non--cube}
\label{ssec:rnres}

From the Euler criterion, a finite field $\mathbb{F}_q$ contains a non--cube element $r$ if and only if $\gcd(3, q - 1) > 1 \Leftrightarrow q \equiv 1 \pmod{3}$, so that $(q - 1) / 3 = \lfloor q / 3 \rfloor$ and
\begin{align*}
  \begin{cases}
  r^{(q - 1) / 3} \neq 1, \\
  r^{q - 1} = 1 \end{cases}
  \Leftrightarrow r^{\lfloor q / 3 \rfloor} = \omega,
  \text{ primitive cubic root of unity.}
\end{align*}

In this case, the polynomial $t^3 - r$ is irreducible over $\mathbb{F}_q$, so that
\begin{align*}
  \mathcal{R}_r = \mathbb{F}_q[t] / \langle t^3 - r \rangle \cong \mathbb{F}_{q^3}.
\end{align*}
We can obtain a result analogous to \cref{th:mv1}.

\begin{theo}
If $r$ is a non--cube in $\mathbb{F}_q$, then $(\mathcal{C}_r, \odot_r)$ is a cyclic group of order $q^2 + q + 1$.
\end{theo}

\begin{proof}
We clearly have that $\mathcal{R}_r^{\odot_r} \cong \mathbb{F}_{q^3}^\times$ has $q^3 - 1$ elements.
If $G \subset \mathbb{F}_{q^3}^\times$ denotes the multiplicative subgroup of order $q^2 + q + 1$, then $x + y t + z t^2 \in G$ if and only if $(x + y t + z t^2)^{q^2 + q + 1} = 1$ and 
\begin{align*}
  (x + y t + z t^2)^{q^2 + q + 1} 
  & = (x + y t + z t^2)^{q^2} (x + y t + z t^2)^q (x + y t + z t^2) \\
  & = (x + y t^q + z t^{2 q})^q (x + y t^q + z t^{2 q}) (x + y t + z t^2),
\end{align*}
where
\begin{align*}
  t^q = (t^3)^{(q - 1)/3} t 
  = r^{\lfloor q/3 \rfloor} t 
  = \omega t, 
  \quad \omega^q 
  = (\omega^3)^{(q - 1)/3} \omega = \omega,
\end{align*}
so that
\begin{align*}
  (x + y t + z t^2)^{q^2 + q + 1} 
  & = (x + y \omega t + z \omega^2 t^2)^q (x + y \omega t + z \omega^2 t^2) (x + y t + z t^2) \\
  & = (x + y \omega^q t^q + z \omega^{2 q} t^{2 q}) (x + y \omega t + z \omega^2 t^2) (x + y t + z t^2) \\
  & = (x + y \omega^2 t + z \omega t^q) (x + y \omega t + z \omega^2 t^2) (x + y t + z t^2) \\
  & = x^3 - 3 r x y z + r y^3 + r^2 z^3.
\end{align*}
Thus, $x + y t + z t^2 \in G \Leftrightarrow (x, y, z) \in \mathcal{C}_r$.
This association is a group isomorphism between $G$ and $(\mathcal{C}_r, \odot_r)$, hence the Pell cubic is cyclic with order $q^2 + q + 1$.
\end{proof}

Looking at the projectivization $\mathbb{P}_r$, since there are no cubic roots of $r$ in $\mathbb{F}_q$, then $\#\mathbb{P}_r = q^2 + q + 1$  from \cref{eq:cubproj1}.
This is obtained also considering that
\begin{align*}
  \big(\mathbb{P}_r, \odot_r\big)
  \cong \mathcal{R}_r^{\odot_r} / \mathbb{F}_q^\times
  \cong \mathbb{F}_{q^3}^\times/\mathbb{F}_q^\times,
\end{align*}
which proves also that $(\mathbb{P}_r, \odot_r)$ is cyclic because quotient of cyclic groups.
In addition, it is possible to obtain the following result.

\begin{theo}
If $q \equiv 1 \pmod{3}$ and $r \in \mathbb{F}_q^\times$ is a non--cube, then there is the group isomorphism
\begin{align*}
  \psi_1 : \big(\mathbb{P}_r, \odot_r\big) 
  & \xrightarrow{\;\sim\;} \big(\mathcal{C}_r, \odot_r\big), \\
  [l : m : n] 
  & \longmapsto N_r(l, m, n)^{\lfloor q / 3 \rfloor - 1} (l, m, n)^{\odot_r 3}.
\end{align*}
\end{theo}

\begin{proof}
In order for $\psi_1$ to be a group isomorphism, it must be:
\begin{itemize} 
\item well defined: 
$\lfloor q / 3 \rfloor - 1 = (q - 4) / 3$ so that for any $[l : m : n] \in \mathbb{P}_r$, $\lambda \in \mathbb{F}_q^\times$
\begin{align*}
  \psi_1([\lambda l : \lambda m : \lambda n]) 
  & = \big(\lambda^3 N_r(l, m, n)\big)^{(q - 4) / 3} \big(\lambda^3 (l, m, n)^{\odot_r 3}\big) \\
  & = \lambda^{q - 1} \psi_1([l : m : n]) = \psi_1([l : m : n]),
\end{align*}
and $\psi_1\big(\mathbb{P}_r\big) \subseteq \mathcal{C}_r$ because for any $[l : m : n] \in \mathbb{P}_r$
\begin{align*}
  N_r\big(\psi_1([l : m : n])\big) 
  & = N_r(l, m, n)^{q - 4} N_r(l, m, n)^3 \\
  & = N_r(l, m, n)^{q - 1} = 1;
\end{align*}

\item a group homomorphism:
given $[l_1 : m_1 : n_1], [l_2 : m_2 : n_2] \in \mathbb{P}_r$,
\begin{align*}
  \psi_1([l_1\!:\!m_1\!:\!n_1]\!\odot_r\![l_2\!:\!m_2\!:\!n_2]) 
  & = N_r([l_1 : m_1 : n_1] \odot_r [l_2 : m_2 : n_2])^{\lfloor q / 3 \rfloor - 1} \\
  & \qquad \big([l_1 : m_1 : n_1] \odot_r [l_2 : m_2 : n_2]\big)^{\odot_r 3} \\
  & = N_r(l_1, m_1, n_1)^{\lfloor q / 3 \rfloor - 1} (l_1, m_1, n_1)^{\odot_r 3} \\
  & \qquad \odot_r N_r(l_2, m_2, n_2)^{\lfloor q / 3 \rfloor - 1} (l_2, m_2, n_2)^{\odot_r 3} \\
  & = \psi_1([l_1 : m_1 : n_1]) \odot_r \psi_1([l_2 : m_2 : n_2]);
\end{align*}

\item injective:
for any $[l : m : n] \in \mathbb{P}_r$, $\psi_1([l : m : n]) = (1, 0, 0)$ if and only if 
\begin{align*}
  \begin{cases}
  N_r(l, m, n)^{\lfloor q / 3 \rfloor - 1} (l^3\!+\!6 r l m n\!+\!r m^3\!+\!r^2 n^3) = 1, \\
  N_r(l, m, n)^{\lfloor q / 3 \rfloor - 1} (3 l^2 m + 3 r l n^2 + 3 r m^2 n) = 0, \\
  N_r(l, m, n)^{\lfloor q / 3 \rfloor - 1} (3 l^2 n + 3 l m^2 + 3 r m n^2) = 0, \end{cases} 
\end{align*}
with $N_r(l, m, n) \neq 0$, so that:

\begin{itemize}
\item if $m, n \neq 0$, then
\begin{align*}
  \begin{cases}
  l^2 m n + r l n^3 + r m^2 n^2 = 0, \\
  l^2 m n + l m^3 + r m^2 n^2 = 0 \end{cases}
  \Leftrightarrow l (r n^3 - m^3) = 0  
  \Leftrightarrow l = 0,
\end{align*}
since $r$ is not a cube.
However, this implies $m = 0$ or $n = 0$;

\item if $m \neq n = 0$, then from the third equation $l m^2 = 0$, i.e., $l = 0$, so that $[l : m : n] = [0 : 1 : 0]$ and the first equation remains $r^{\lfloor q / 3 \rfloor} = 1$, which is not true because of the generalized Euler criterion;

\item if $n \neq m = 0$, then from the second equation $r l n^2 = 0$, i.e., $l = 0$, so that $[l : m : n] = [0 : 0 : 1]$ and the first equation remains $r^{2 \lfloor q / 3 \rfloor} = 1$.
This means $r^{(q - 1) / 3} = \pm 1$, but $r^{(q - 1) / 3} = - 1$ is not valid since it implies $r^{q - 1} = - 1$, while $r^{(q - 1) / 3} = 1$ is an absurd for the generalized Euler criterion;

\item $m = n = 0$ is the only remaining option, i.e., $\ker(\psi_1) = \{[1 : 0 : 0]\}$;
\end{itemize}

\item surjective:
this is straightforward because it is an injection between two finite groups of the same cardinality $q^2 + q + 1$.
\end{itemize}
In conclusion, $\psi_1$ is a group isomorphism.
\end{proof}

Thus, using the group isomorphism $\psi_1$ also proves that $(\mathcal{C}_r, \odot_r)$ is cyclic of order $q^2 + q + 1$.
This construction allows also to find all the solutions of the cubic Pell equation.
Indeed, it is sufficient to evaluate $\psi_1$ over all the elements of $\mathbb P_r$, which are $[l: m: 1]$ for all $l, m \in \mathbb F_q$, $[l: 1: 0]$ for all $l \in \mathbb F_q$ and $[1: 0: 0]$, as obtained in \cref{eq:cubproj1}.
However, since the explicit inverse is missing, it is difficult to describe each point of the Pell cubic as a point of the projectivization.

\begin{es}
Let us consider $q = 7$ and $r = 2$, which is not a cube in $\mathbb{F}_7$.
Thanks to the previous results we know that the cubic Pell equation
\begin{equation*}
    x^3 + 2 y^3 + 4 z^3 - 6 x y z \equiv 1 \pmod{7},
\end{equation*}
admits $q^2 + q + 1 = 57$ solutions and we are able to find all of them evaluating 
\begin{gather*}
    \psi_1([l : m : 1]), \quad \forall\,l, m \in \mathbb{F}_7,\\
    \psi_1([l : 1 : 0]), \quad \forall\,l \in \mathbb{F}_7,\\
    \psi_1([1 : 0 : 0]) = (1, 0, 0).
\end{gather*}
For instance, for finding a random solution of the cubic Pell equation, we can take two random elements $l, m \in \mathbb F_7$, e.g., $l = 3$ and $m = 5$ and evaluate
\begin{equation*}
    \psi_1([3 : 5 : 1]) = (5, 4, 4).
\end{equation*}
One can check that 
\begin{equation*}
    5^3 + 2 \cdot 4^3 + 4 \cdot 4^3 - 6 \cdot 5 \cdot 4 \cdot 4 \equiv 1 \pmod{7}.
\end{equation*}
Similarly, we can take $l = 4$ and $[4 : 1 : 0] \in \mathbb{P}_2$, so that
\begin{equation*}
    \psi_1([4: 1: 0]) = (2, 4, 1),
\end{equation*}
is another solution of the cubic Pell equation.
\end{es}

Note that for large values of $q$ this method for finding all the solutions of the cubic Pell equation is not efficient, since it has complexity $O(q^2)$, even if it is surely better than an exhaustive search that has complexity $O(q^3)$.

However, for large values of $q$ it is really interesting to use the above method for generating random solutions of the cubic Pell equation since, exploiting $\psi_1$ as in the previous example, we are always able to generate different solutions.

Without our method, the probability that a random triple $(x, y, z) \in \mathbb{F}_q^3$ is a solution of the cubic Pell equation is $\frac{q^2 + q + 1}{q^3}$ that is very low when $q$ is large. 
It is theoretically possible, given a random element $(x, y, z) \in \mathbb{F}_q^3$, to obtain a solution of the cubic Pell equation taking 
\begin{equation*}
    \left( \frac{x}{\sqrt[3]{N_r(x,y,z)}}, \frac{y}{\sqrt[3]{N_r(x,y,z)}}, \frac{z}{\sqrt[3]{N_r(x,y,z)}} \right),
\end{equation*}
but evaluating the cubic root of an element in a finite field is an hard problem for large values of $q$.

\subsection{$r$ cube with three roots in $\mathbb{F}_q$}
\label{ssec:r3roots}

If $q \equiv 1 \pmod{3}$, given $\omega$ primitive cubic root of unity, then $\{1, \omega, \omega^2\} \subset \mathbb{F}_q$.
In addition, if $r$ is a cube, fixed a cubic root $s \in \mathbb{F}_q^\times$ of $r$, then  the other two cubic roots are $\omega s, \omega^2 s$ and $\{s, \omega s, \omega^2 s\} \subseteq \mathbb{F}_q^\times$.
In this case, with a proof analogous to \cref{th:mv2}, we prove the following result.

\begin{theo}
If $q \equiv 1 \pmod{3}$ and $r \in \mathbb{F}_q^\times$ is a cube, then $(\mathcal{C}_r, \odot_r)$ is isomorphic to $\mathbb{F}_q^\times \times \mathbb{F}_q^\times$.
\end{theo}

\begin{proof}
Fixed a cubic root $s \in \mathbb{F}_q^\times$ of $r$, the norm of a point $(x, y, z) \in \mathcal{C}_r$ can be written as
\begin{align*}
  1 & = x^3 - 3 r x y z + r y^3 + r^2 z^3 \\
  & = (x + \omega s y + \omega^2 s^2 z) (x + \omega^2 s y + \omega s^2 z) (x + s y + s^2 z) 
  = u v w,
\end{align*}
so that
\begin{align*}
  x = \frac{w + v + u}{3}, 
  \quad y = \frac{w + \omega v + \omega^2 u}{3 s}, 
  \quad z = \frac{w + \omega^2 v + \omega u}{3 s^2},
\end{align*}
is a bijective correspondence between the points $(x, y, z) \in \mathcal{C}_r$ and $(u, v, w) \in \mathbb{F}_q^3$ such that $u v w = 1$.
This equation has exactly $(q - 1)^2$ solutions in $\mathbb{F}_q^3$ and, in particular, a unique solution for each $(u, v) \in \mathbb{F}_q^\times \times \mathbb{F}_q^\times$.
Thus,
\begin{gather*}
    \big(\mathcal{C}_r, \odot_r\big) 
    \;\cong\; \mathbb{F}_q^\times \times \mathbb{F}_q^\times \\
  (x, y, z) \longmapsto (x + \omega s y + \omega^2 s^2 z, x + \omega^2 s y + \omega s^2 z), \\
  \left(\frac{1 + u v^2 + u^2 v}{3 u v}, 
  \frac{1 + \omega u v^2 + \omega^2 u^2 v}{3 s u v}, 
  \frac{1 + \omega^2 u v^2 + \omega u^2 v}{3 s^2 u v}, \right)
  \longmapsfrom (u, v),
\end{gather*}
is bijective and also a group homomorphism.
\end{proof}

When considering the projectivization $\mathbb{P}_r$, it is clear from \cref{eq:cubproj2} that 
\begin{align*}
  \#\mathbb{P}_r = q^2 + q + 1 - 3 q = (q - 1)^2.
\end{align*}
This is confirmed by the following result, obtained analogously to \cref{th:conpr}.

\begin{theo}
If $q \equiv 1 \pmod{3}$ and $r \in \mathbb{F}_q^\times$ is a cube, then $(\mathbb{P}_r, \odot_r)$ is isomorphic to $\mathbb{F}_q^\times \times \mathbb{F}_q^\times$.
\end{theo}

\begin{proof}
Fixed $s$ cubic root of $r$ in $\mathbb{F}_q$, $t^3 - r$ is reducible over $\mathbb{F}_q$ as
\begin{align*}
  t^3 - r = (t - s) (t - \omega s) (t - \omega^2 s),
\end{align*}
so that, using the Chinese remainder theorem, there is the ring isomorphism
\begin{align*}
  \mathcal{R}_r = \mathbb{F}_q[t] / \langle t^3 - r \rangle
  & \xrightarrow{\;\sim\;} \mathbb{F}_q[t] / \langle t - s \rangle \times \mathbb{F}_q[t] / \langle t - \omega s \rangle \times \mathbb{F}_q[t] / \langle t - \omega^2 s \rangle, \\
  x + y t + z t^2 
  & \longmapsto (x + s y + s^2 z, x + \omega s y + \omega^2 s^2 z, x + \omega^2 s y + \omega s^2 z).
\end{align*}
In addition, $\mathbb{F}_q[t] / \langle t - s \rangle \cong \mathbb{F}_q[t] / \langle t - \omega s \rangle \cong \mathbb{F}_q[t] / \langle t - \omega^2 s \rangle \cong \mathbb{F}_q$, and when passing to the quotients there is the group isomorphism
\begin{align*}
  \big(\mathbb{P}_r, \odot_r\big)
  \cong \mathcal{R}_r^{\odot_r}/\mathbb{F}_q^\times
  & \;\cong\; (\mathbb{F}_q^\times \times \mathbb{F}_q^\times \times \mathbb{F}_q^\times) / \mathbb{F}_q^\times
  \cong \mathbb{F}_q^\times \times \mathbb{F}_q^\times, \\
  [l : m : n] = \{\lambda (l + m t + n t^2) \,|\, \lambda \neq 0\}
  & \mapsto \left(\!\frac{l\!+\!\omega s m\!+\!\omega^2\!s^2 n}{l + s m + s^2 n},\!
  \frac{l\!+\!\omega^2\!s m\!+\!\omega s^2 n}{l + s m + s^2 n}\!\right)\!, \\
  [s^2\!(1\!+\!v\!+\!u)\!:\!s (1\!+\!\omega v\!+\!\omega^2\!u)\!:\!1\!+\!\omega^2\!v\!+\!\omega u]
  & \mapsfrom (u, v).
\end{align*}
\end{proof}

Combining the obtained results gives the explicit group isomorphism
\begin{align*}
  \psi_2\!:\!\big(\mathbb{P}_r,\!\odot_r\big)\!\xrightarrow{\sim} 
  & \big(\mathcal{C}_r, \odot_r\big), \\
  [l : m : n] \mapsto 
  & \left(\frac{l^3 + 2 s^2 l (m^2 + s m n + s^2 n^2) + s^4 m n (m + s n)}{N_r(l, m, n)},\right.\\
  & \qquad\quad \frac{s^2 m^3 + 2 m (l^2 + s^2 l n + s^4 n^2) + s l n (l + s^2 n)}{N_r(l, m, n)}, \\
  & \qquad\qquad \left.\frac{s^5 n^3 + 2 s n (l^2 + s l m + s^2 m^2) + l m (l + s m)}{s N_r(l, m, n)}\right),
\end{align*}
where the sum of the numerators is $(l + s m + s^2 n)^3$. 
The inverse is given by
\begin{align*}
  \psi_2^{- 1} : \big(\mathcal{C}_r, \odot_r\big) \xrightarrow{\sim} 
  & \big(\mathbb{P}_r, \odot_r\big), \\
  (x, y, z) \mapsto 
  & \big[s^2 (1\!+\!2 x\!-\!s y\!-\!s^2 z) : s (1\!-\!x\!+\!2 s y\!-\!s^2 z) : 1\!-\!x\!-\!s y\!+\!2 s^2 z\big].
\end{align*}

The group isomorphism $\psi_2$ allows to find all the solutions of the cubic Pell equation: it is sufficient to evaluate $\psi_2$ over all the elements of $\mathbb P_r$ described explicitly in \cref{eq:cubproj2}.
In addition, differently from the previous case, the explicit inverse of the group isomorphism can be used to describe each point of the Pell cubic with two thirds of the size with respect to the classic notation for the points in $\mathbb{F}_q^3$.

\begin{es}
Let us consider $q = 13$ and $r = 5$, which is the cube of $\{7, 8, 11\}$ in $\mathbb{F}_{13}$.
Thanks to the previous results we know that the cubic Pell equation
\begin{equation*}
    x^3 + 5 y^3 - z^3 - 2 x y z \equiv 1 \pmod{13},
\end{equation*}
admits $(q - 1)^2 = 144$ solutions and we are able to find all of them evaluating 
\begin{gather*}
    \psi_2([l : m : 1]), \quad \forall\, m \in \mathbb{F}_{13},\, l \in \mathbb{F}_{13} \smallsetminus \{- 7 m + 3, - 8 m + 1, - 11 m + 9\}, \\
    \psi_2([l : 1 : 0]), \quad \forall\,l \in \mathbb{F}_{13} \smallsetminus \{- 7, - 8, - 11\},\\
    \psi_2([1 : 0 : 0]) = (1, 0, 0).
\end{gather*}
For instance, for finding a random solution of the cubic Pell equation, we can take a random $m \in \mathbb{F}_{13}$, e.g., $m = 3$, and another element $l \in \mathbb{F}_{13} \smallsetminus \{8, 3, 2\}$, e.g., $l = 9$, and evaluate
\begin{equation*}
    \psi_2([9 : 3 : 1]) = (3, 4, 3).
\end{equation*}
One can check that 
\begin{equation*}
    3^3 + 5 \cdot 4^3 - 3^3 - 2 \cdot 3 \cdot 4 \cdot 3 \equiv 1 \pmod{13}.
\end{equation*}
Similarly, we can take $l = 4 \not\in \{6, 5, 2\}$ and $[4 : 1 : 0] \in \mathbb{P}_5$, so that
\begin{equation*}
    \psi_1([4: 1: 0]) = (10, 4, 9),
\end{equation*}
is another solution of the cubic Pell equation.
\end{es}

\subsection{$r$ cube with one root in $\mathbb{F}_q$} 
\label{ssec:r1root}

If $q \not\equiv 1 \pmod{3}$, then $\mathbb{F}_q$ does not contain any non--trivial cubic root of unity.
In addition, each $r \in \mathbb{F}_q^\times$ is a cube and has only one cubic root $s$ in $\mathbb{F}_q$.

In this case, \cref{eq:cubproj3} holds and the projectivization $\mathbb{P}_r$ has
\begin{align*}
  \#\mathbb{P}_r = q^2 + q + 1 - (q + 2) = q^2 - 1,
\end{align*}
unless there is a $m \in \mathbb{F}_q$ such that $[- (m + s) s : m : 1] = [s^2 : s : 1] \Leftrightarrow 3 s^2 = 0$, which is satisfied only when $q = 3^k$, in which case $\#\mathbb{P}_r = q^2$.
This result is also confirmed by the following statement, obtained analogously to \cref{th:conpr}.

\begin{theo}
If $q \not\equiv 1 \pmod{3}$, $q \neq 3$ and $r \in \mathbb{F}_q^\times$, then $(\mathbb{P}_r, \odot_r)$ is a cyclic group of order $q^2 - 1$.
\end{theo}

\begin{proof}
Given $s$ cubic root of $r$ in $\mathbb{F}_q$, $t^3 - r$ is reducible over $\mathbb{F}_q$ as
\begin{align*}
  t^3 - r = (t - s) (t^2 + s t + s^2),
\end{align*}
so that, using the Chinese remainder theorem, there is the ring isomorphism
\begin{align*}
  \mathcal{R}_r = \mathbb{F}_q[t] / \langle t^3 - r \rangle
  & \xrightarrow{\;\sim\;} \mathbb{F}_q[t] / \langle t - s \rangle \times \mathbb{F}_q[t] / \langle t^2 + s t + s^2 \rangle, \\
  x + y t + z t^2 
  & \longmapsto \big(x + s y + s^2 z, x - s^2 z + (y - s z) t\big).
\end{align*}
In addition, $\mathbb{F}_q[t] / \langle t - s \rangle \cong \mathbb{F}_q$ and $\mathbb{F}_q[t] / \langle t^2 + s t + s^2 \rangle \cong \mathbb{F}_{q^2}$, and when passing to the quotients there is the group isomorphism
\begin{align*}
  \big(\mathbb{P}_r, \odot_r\big) 
  \cong \mathcal{R}_r^{\odot_r}/\mathbb{F}_q^\times
  & \;\cong\; (\mathbb{F}_q^\times \times \mathbb{F}_{q^2}^\times) / \mathbb{F}_q^\times
  \cong \mathbb{F}_{q^2}^\times, \\
  [l : m : n] = \{\lambda (l + m t + n t^2) \,|\, \lambda \neq 0\}
  & \longmapsto \left(\frac{l - s^2 n}{l + s m + s^2 n}, \frac{m - s n}{l + s m + s^2 n}\right), \\
  [s^2 (1\!-\!s v\!+\!2 u)\!:\!s (1\!+\!2 s v\!-\!u)\!:\!1\!-\!s v\!-\!u]
  & \longmapsfrom (u, v).
\end{align*}
\end{proof}

The relation with the Pell cubic when $p \neq 3$ is given by the following result.

\begin{theo}
If $q \equiv 2 \pmod{3}$ and $r \in \mathbb{F}_q^\times$, then there is the group isomorphism
\begin{align*}
  \psi_3 : \big(\mathbb{P}_r, \odot_r\big) 
  & \xrightarrow{\;\sim\;} \big(\mathcal{C}_r, \odot_r\big), \\
  [l : m : n] 
  & \longmapsto N_r(l, m, n)^{\lfloor q / 3 \rfloor} (l, m, n).
\end{align*}
\end{theo}

\begin{proof}
In order for $\psi_3$ to be a group isomorphism, it must be:
\begin{itemize}
\item well defined:
$\lfloor q / 3 \rfloor = (q - 2) / 3$ so that for any $[l : m : n] \in \mathbb{P}_r$, $\lambda \in \mathbb{F}_q^\times$
\begin{align*}
  \psi_3\big([\lambda l : \lambda m : \lambda n]\big) 
  & = \big(\lambda^3 N_r(l, m, n)\big)^{(q - 2) / 3} \lambda (l, m, n) \\
  & = \lambda^{q - 1} \psi_3([l : m : n]) 
  = \psi_3([l : m : n]),
\end{align*}
and $\psi_3\big(\mathbb{P}_r\big) \subseteq \mathcal{C}_r$ because for any $[l : m : n] \in \mathbb{P}_r$
\begin{align*}
  N_r\big(\psi_3([l : m : n])\big) 
  & = N_r(l, m, n)^{q - 2} N_r(l, m, n) \\
  & = N_r(l, m, n)^{q - 1} = 1;
\end{align*}

\item a group homomorphism:
given $[l_1 : m_1 : n_1], [l_2 : m_2 : n_2] \in \mathbb{P}_r$,
\begin{align*}
  \psi_3([l_1 : m_1 : n_1] \odot_r [l_2 : m_2 : n_2]) 
  & = N_r([l_1 : m_1 : n_1] \odot_r [l_2 : m_2 : n_2])^{\lfloor q / 3 \rfloor} \\
  & \qquad [l_1 : m_1 : n_1] \odot_r [l_2 : m_2 : n_2] \\
  & = N_r(l_1, m_1, n_1)^{\lfloor q / 3 \rfloor} (l_1, m_1, n_1) \\
  & \qquad \odot_r N_r(l_2, m_2, n_2)^{\lfloor q / 3 \rfloor} (l_2, m_2, n_2) \\
  & = \psi_3([l_1 : m_1 : n_1]) \odot_r \psi_3([l_2 : m_2 : n_2]);
\end{align*}

\item injective:
for any $[l : m : n] \in \mathbb{P}_r$, $N_r(l, m, n) \neq 0$ and
\begin{align*}
  \psi_3([l : m : n]) = (1, 0, 0) 
  & \Leftrightarrow \begin{cases}
  N_r(l, m, n)^{\lfloor q / 3 \rfloor} l = 1, \\
  N_r(l, m, n)^{\lfloor q / 3 \rfloor} m = 0, \\
  N_r(l, m, n)^{\lfloor q / 3 \rfloor} n = 0 \end{cases} \\
  & \Leftrightarrow \begin{cases}
  (l^3)^{(q - 2) / 3} l = 1, \\
  m = 0, \\
  n = 0, \end{cases} \\
  & \Leftrightarrow [l : m : n] = [1 : 0 : 0];
\end{align*}

\item surjective:
\begin{itemize}
\item $x^3 = 1$ admits only the solution $x = 1$, so that $(1, 0, 0)$ is the only point of $\mathcal{C}_r$ with $y = z = 0$, as well as $[1 : 0 : 0]$ in $\mathbb{P}_r$;

\item if $z = 0$ but $y \neq 0$, then the preimage is of the form $[l : 1 : 0]$ and
\begin{align*}
  \begin{cases}
  x = (l^3 + r)^{\lfloor q / 3 \rfloor} l, \\
  y = (l^3 + r)^{\lfloor q / 3 \rfloor} \end{cases} 
  \Rightarrow l = \frac{x}{y};
\end{align*} 

\item if $z \neq 0$ then 
\begin{align*}
  \begin{cases}
  x = N_r(l, m, 1)^{\lfloor q / 3 \rfloor} l, \\
  y = N_r(l, m, 1)^{\lfloor q / 3 \rfloor} m, \\
  z = N_r(l, m, 1)^{\lfloor q / 3 \rfloor} \end{cases} 
  \Rightarrow \begin{cases}
  l = x / z, \\
  m = y / z.\end{cases}
\end{align*}
\end{itemize}
\end{itemize}
In conclusion, $\psi_3$ is a group isomorphism with inverse the classic projectivization 
\begin{align*}
  \psi_3^{-1} : \big(\mathcal{C}_r, \odot_r\big) 
  & \xrightarrow{\;\sim\;} \big(\mathbb{P}_r, \odot_r\big), \\
  (1, 0, 0) & \longmapsto [1 : 0 : 0], \\
  (x, y, 0) & \longmapsto [x / y : 1 : 0], \\
  (x, y, z) & \longmapsto [x / z : y / z : 1].
\end{align*}
\end{proof}

\begin{rem}
When $q = p ^k$ with $p = 3$, with an analogous proof, the obtained group isomorphism is 
\begin{align*}
  \psi_3' : \big(\mathbb{P}_r, \odot_r\big) 
  & \xrightarrow{\;\sim\;} \big(\mathcal{C}_r, \odot_r\big), \\
  [l : m : n] 
  & \longmapsto N_r(l, m, n)^{q / 3 - 1} (l, m, n)^{\odot_r 2}.
\end{align*}
\end{rem}

Thanks to the group isomorphism $\psi_3$, the properties of $\big(\mathbb{P}_r, \odot_r\big)$ are inherited by $\big(\mathcal{C}_r, \odot_r\big)$, i.e, it is cyclic with $q^2 - 1$ elements.
In addition, it allows to find all the solutions of the cubic Pell equation by simply evaluating $\psi_3$ over all the elements of $\mathbb P_r$, which are described explicitly in \cref{eq:cubproj3}.
As in the previous case, the explicit inverse can be used to describe each point of the Pell cubic with two thirds of the size of points in $\mathbb{F}_q^3$.

\begin{es}
Let us consider $q = 11$ and $r = 9$, which is the cube of $4$ in $\mathbb{F}_{11}$.
Thanks to the previous results we know that the cubic Pell equation
\begin{equation*}
    x^3 + 9 y^3 + 4 z^3 + 6 x y z \equiv 1 \pmod{11},
\end{equation*}
admits $q^2 - 1 = 120$ solutions and we are able to find all of them evaluating 
\begin{gather*}
    \psi_3([l : m : 1]), \quad \forall\, m \in \mathbb{F}_{11},\, l \in \mathbb{F}_{11} \smallsetminus \{- 4 m + 5\},\, (l, m) \neq (5, 4), \\
    \psi_3([l : 1 : 0]), \quad \forall\,l \in \mathbb{F}_{11} \smallsetminus \{- 4\},\\
    \psi_3([1 : 0 : 0]) = (1, 0, 0).
\end{gather*}
For instance, for finding a random solution of the cubic Pell equation, we can take a random $m \in \mathbb{F}_{11}$, e.g., $m = 2$, and another element $l \in \mathbb{F}_{11} \smallsetminus \{8\}$, e.g., $l = 7$, and evaluate
\begin{equation*}
    \psi_3([7 : 2 : 1]) = (9, 1, 6).
\end{equation*}
One can check that $9^3 + 9 \cdot 1^3 + 4 \cdot 6^3 + 6 \cdot 9 \cdot 1 \cdot 6 \equiv 1 \pmod{11}$.

Similarly, we can take $l = 3 \neq 7$ and $[3 : 1 : 0] \in \mathbb{P}_9$, so that
\begin{equation*}
    \psi_1([3: 1: 0]) = (4, 5, 0),
\end{equation*}
is another solution of the cubic Pell equation.
\end{es}

\bibliographystyle{abbrv}
\bibliography{biblio}

\end{document}